\documentclass[a4paper]{amsart}

\usepackage{amsmath,amsthm,amsfonts,amssymb, mathdots}
\usepackage{epsfig, psfrag}
\usepackage{graphicx,graphics}
\usepackage{latexsym}
\usepackage[all]{xy}
\usepackage{subfigure}
\usepackage[usenames,dvips]{color}  
\usepackage{color} 
\usepackage{psfrag}  

\xyoption{matrix}
\xyoption{arrow}

\theoremstyle{plain}
\newtheorem{lemma}{Lemma}[section]
\newtheorem{prop}[lemma]{Proposition}
\newtheorem{cor}[lemma]{Corollary}
\newtheorem{thm}[lemma]{Theorem}

\newtheorem{de}[lemma]{Definition}
\newtheorem{ex}[lemma]{Example}
\newtheorem{rem}[lemma]{Remark}

\newcommand{\Z}{\mathbb Z}

\newcommand{\N}{\mathbb N}

\newcommand{\diagramnode}[1]{\makebox[1.5em]{#1}}
\begin{document}

\title[]  
{A geometric interpretation of the triangulated structure of $m$-cluster categories} 

\author{Lamberti Lisa}


\keywords{m-cluster categories, triangulated
structure, Auslander-Reiten triangles, orbit categories, Auslander Reiten quivers, m-th power of a translation quiver} %

\begin{abstract}
The aim of this note is to answer several open problems arising
from the geometric description of the $m$-cluster categories of type $A_n$
and their realization in terms of the $m$-th power
of a translation quiver. In particular, we give a geometric interpretation of the triangulated
structure of $m$-cluster categories. Furthermore, we characterize
all the connected components arising from a cluster category when taking the 
$m$-th power of its Auslander-Reiten quiver.
\end{abstract}

\maketitle

\section{Introduction}
Cluster categories were introduced in \cite{ccs} and \cite{BMRRT} with the aim of modeling the features of cluster algebras,
discovered by S. Fomin and A. Zelevinsky in \cite{FZ}. 
In \cite{BMRRT} they were defined as orbit categories of the bounded derived category of finitely generated
modules over a finite dimensional hereditary algebra $H$ over a field $k$ by a certain auto-equivalence.
A number of cluster categories
are approachable in a more concrete way. In fact, in
\cite{ccs} the cluster categories associated
to the path algebras $H:=kQ$ over a Dynkin quiver of type $A_n$ were expressed in geometric terms by means of 
triangulations of a certain polygon $\Pi$. 
The objects of the cluster category of \cite{ccs} are direct sums of diagonals and the morphisms are given 
by certain rotations around vertices of $\Pi$
modulo the so-called mesh relations which allow the exchange of certain types of rotations.
Similarly, R. Schiffler provided in \cite{Sch} a geometric description of cluster categories of type $D_n$ 
by means of polygons which have a single puncture in their centers. More geometric models have been given in
\cite{bm3}, \cite{bz} and \cite{to}.

Cluster categories have a number of interesting properties: they are triangulated,
the projection functor from $\mathcal{D}^b(\mathrm{mod}H)$ to the cluster category is a triangle-functor and they are
Calabi-Yau of dimension $2$. This was all proven by B. Keller in \cite{K1}. Furthermore, these categories
are Krull-Schmidt and have Auslander-Reiten triangles \cite{BMRRT}.
In addition, the geometric model is beautifully linked to cluster algebras:
in type $A_n$ we have the bijection between the 
isomorphism classes of indecomposable objects of $\mathcal{C}:=\mathcal{C}_{A_n}$ and the diagonals of $\Pi$
as well as between the diagonals of $\Pi$ and cluster variables of a cluster algebra of type $A_n$. 
Under the latter bijection maximal sets of non-crossing diagonals correspond to clusters, and mutations of
cluster variables can be interpreted 
as flips of diagonals.

As a generalization of cluster categories, B. Keller introduced
in \cite{K1} the so called $m$-cluster categories $\mathcal{C}^m,$ for $m\in\N$. 
These categories, studied for example by 
\cite{bm1}, \cite{bm2}, \cite{Jo}, \cite{th}, \cite{to}, \cite{wr}, \cite{zh}, \cite{zz}, 
have the same properties as the original ones, except that this time
they are Calabi-Yau of dimension $m+1$ \cite{K1}. Also for certain $m$-cluster categories a geometric
description was provided. In \cite{bm1}, K. Baur and R. Marsh 
expressed the $m$-cluster categories of type 
$A_n$ by means of the so-called $m$-diagonals which are the diagonals dividing a given $(mn+2)$-gon $\Pi$
into two smaller polygons whose number of sides are congruent to $2$ modulo $m$. This construction recovers
the original one given in \cite{ccs} for $m=1$. In \cite{bm2} the authors gave a geometric description
of $m$-cluster categories of type $D_n$
generalizing Schiffler's work \cite{Sch} for cluster category for type $D_n$.

In \cite{bm1}, K. Baur and R. Marsh establish a combinatorial
method enabling one to recover important information about $\mathcal{C}^m$
starting from $\mathcal{C}.$ This information is
recorded in the Auslander-Reiten quivers of the corresponding categories. In fact they showed
that one can obtain the Auslander-Reiten quiver of an $m$-cluster category of type $A_{n-1},$ starting from 
the one of a cluster category of type $A_{nm-1}$. For this, \cite{bm1} introduced the concept of
the $m$-th power of a stable translation quiver, $\Gamma$,
obtained from the set of diagonals of the $nm+2$-gon $\Pi$, with arrows arising from paths in  $\Gamma$.
In general, the $m$-th power of $\Gamma$ 
consists of several connected components. One of them  is
isomorphic to the Auslander-Reiten quiver of an $m$-cluster category of type $A_{n-1}$, as K. Baur and R. Marsh showed.

Concerning the content of this note, our interest is twofold. On one hand, 
we are interested in describing the
triangulated structure of the $m$-cluster categories of type $A_{n-1}$
on the level of the geometric model, and understanding it through the concept of the 
$m$-th power of $\Gamma$. On the other hand, we will give a complete characterization of all
the connected components arising with this procedure.

More precisely, the plan is as follows. We start by recalling the necessary 
definitions and preliminary results, then we present the $m$-cluster categories of type $A_{n-1}$ 
and show how one can obtain their Auslander-Reiten quivers in geometric terms. We will also
recall the concept of the $m$-th power of a translation quiver (Section 2).
In the third section we explain how the triangulated structure of the $m$-cluster category
can be recovered from the polygon $\Pi$. For this we will first 
study Auslander-Reiten triangles of $\mathcal{C}^m$ which turn out to have
a beautiful geometric description. Secondly, we will explain the
link between Auslander-Reiten triangles of $\mathcal{C}^m$ and $\mathcal{C}$.
Here, we view the former category as obtained from the latter
by taking the $m$-th power of its 
Auslander Reiten quivers. Finally,
we study arbitrary triangles coming from a morphism $\mu$
between two objects in $\mathcal{C}^m$ and provide a geometric description of 
the third object of the triangle of $\mu$, i.e. of $\mathrm{Cone}(\mu)\in\mathcal{C}^m.$ 
Combining these results will
provide a geometric understanding of the triangulated structure of all orbit
categories arising from $\mathcal{C}$ through the $m$-th power of the Auslander-Reiten quiver of $\mathcal{C}$.
A precise description of the different connected components arising 
will be given in Section 4.
In the last section we illustrate on two examples the results of Section 4. 
\newline

\textbf{Acknowledgments:} I would like to thank Prof. K. Baur for introducing me to such an interesting field, 
for the inspiring discussions we had and
for her time in reading my drafts.

\section{Definitions and preliminaries}

In this section we introduce some terminology following mainly \cite{bm1},  \cite{BMRRT}, \cite{ccs} and \cite{K1}.\newline

Throughout this paper let $n,m\in\N$ and fix an $N:=nm+2$-sided regular polygon $\Pi$ with vertices numbered clockwise. 
Consider all operations on the vertices of
$\Pi$ modulo $nm+2$. Unless stated otherwise $Q$ will always be a Dynkin quiver of type $A_{n-1}$ and $\mathcal{D}^b(\mathrm{mod }kQ)$
the bounded derived category of finite-dimensional modules over the path algebra $kQ,$
where $k$ is an algebraically closed field.\newline

The $m$-cluster category of type $A_{n-1}$ is defined as the orbit category of the bounded derived category under 
the action of the cyclic group generated by the auto-equivalence $\tau^{-1}\Sigma^m$:
$$
\mathcal{C}^m:=\mathcal{D}^b(\mathrm{mod }kQ)/(\tau^{-1}\Sigma^m)^{\Z},
$$
where $\tau$ denotes the Auslander-Reiten
translation and $\Sigma^m$ is the composition of the shift functor $\Sigma$ with itself $m$ times.
The isoclasses of objects of $\mathcal{C}^m$ are the $\tau^{-1}\Sigma^m$-orbits $\widetilde{X}:=((\tau^{-1}\Sigma^m)^i X)_{i\in \Z}$
of objects $ X \in \mathcal{D}^b(\mathrm{mod }kQ)$. For two objects $\widetilde{X},\widetilde{Y}$ in $\mathcal{C}^m$
we have:
$$
\mathrm{Hom}_{\mathcal{C}^m}(\widetilde{X},\widetilde{Y})=\bigoplus_{i\in\Z}\mathrm{Hom}_{\mathcal{D}^b(\mathrm{mod }kQ) }((\tau^{-1}\Sigma^m)^iX,Y).
$$
For simplicity later on we will omit the tilde and indicate a representative of the preimage of an object $X\in\mathcal{C}^m$  
under the projection functor $\pi:\mathcal{D}^b(\mathrm{mod }kQ)\rightarrow\mathcal{C}^m$ by $\pi^{-1}(X)$.
This category is $k$-linear triangulated (\cite{K1}), Krull-Schmidt and has Auslander Reiten triangles 
(\cite{BMRRT}). We recall that being Krull-Schmidt means that 
each object decomposes uniquely up to isomorphism and permutation of factors into a finite sum of indecomposable objects. The latter are
objects which are neither zero
nor a direct sum of two non-zero objects. 
Furthermore, remark that the shift functor $\Sigma$ and the Auslander-Reiten translate $\tau_m$ in $\mathcal{C}^m$
are induced by the ones in $\mathcal{D}^b(\mathrm{mod }kQ)$.
For $m=1$ we omit the indices in $\mathcal{C}^m$ and $\tau_m$ and simply write  $\mathcal{C}$, $\tau$. 

Recall that the \textit{Auslander-Reiten quiver} of a Krull-Schmidt category 
$\mathcal{K}$ is a quiver
whose vertices are the isomorphism classes of indecomposable
objects of $\mathcal{K}$ and where the number of arrows between two vertices $[X]$ and $[Y]$ is given by the
dimension of the space of \textit{irreducible morphisms} between the indecomposables $X$ and $Y$:
$$\mathrm{Irr}_{\mathcal{K}}(X,Y):= \mathrm{rad}_{\mathcal{K}}(X,Y)/\mathrm{rad}^2_{\mathcal{K}}(X,Y).$$
Here $\mathrm{rad}_{\mathcal{K}}(X,Y)$ is the subspace of $\mathrm{Hom}_{\mathcal{K}}(X,Y)$ 
formed by all non isomorphisms, and 
$\mathrm{rad}^2_{\mathcal{K}}(X,Y)$ is 
the subspace of all non isomorphisms admitting a non trivial factorization.

To $Q$ we can associate also another quiver which we denote by $\Z Q.$ Its vertices are labeled by pairs $(n,i)$ where
$n\in\Z$ and $i$ is  a vertex of $Q$, and if there is an arrow in $Q$ from $i$ to $j$ then we draw an arrow from 
$(n,i)$ to $(n,j)$ and one from $(n,j)$ to $(n+1,i)$. In this way $\Z Q$ is an infinite strip of copies of $Q$.
Furthermore, one can define a translation map $\tau$ on $\Z Q$ by sending $(n,j)$ to $(n-1,j)$, this makes $\Z Q$ a 
stable translation quiver as defined by  C. Riedtmann, \cite{rie}.

In particular, D. Happel showed in \cite{Ha2} that the Auslander-Reiten quiver of $\mathcal{D}^b(\mathrm{mod }kQ)$
is isomorphic to $\Z Q$. From this result it follows that $\mathcal{D}^b(\mathrm{mod }kQ)$ is independent of
the orientation of $Q$, see Section 4.8 in \cite{Ha2}. Furthermore, the Auslander-Reiten quiver of $\mathcal{C}^m$ is isomorphic 
to the quotient $\Z Q/\varphi^m$, see Proposition 1.3 in \cite{BMRRT}
where $\varphi^m$ is the graph automorphism induced by the auto-equivalence $\tau^{-1}\Sigma^m.$
For simplicity in the following we sometimes write $\tau^{-1}\Sigma^m$ instead of $\varphi^m.$

After having reminded the algebraic definition of $m$-cluster categories we will
recall the geometric approach following \cite{ccs} and \cite{bm1}. In this case the Auslander-Reiten
quiver of $\mathcal{C}^m$ admits another description, as follows from
the next theorem due to P. Caldero, F. Chapoton and R. Schiffler (Section 2 in \cite{ccs}) for the case $m=1$ 
and to K. Baur and R. Marsh (Proposition 5.5 in \cite{bm1}) for $m\geq1$. 

Let $\Pi$ be as before.
As usual we refer to a \textit{diagonal} of $\Pi$ as a straight line between two non-adjacent vertices of $\Pi$.
Whereas an \textit{$m$-diagonal} in $\Pi$ is a diagonal which divides $\Pi$ into an $(mj+2)$-gon
 and an $(m(n-j)+2)$-gon where $j=1,..,\lceil\frac{n-1}{2}\rceil:=\mathrm{min}\{l\in\Z|l\geq\frac{n-1}{2}\}.$
To the $m$-diagonals of $\Pi$ one can associate a quiver $\Gamma^m_{A_{n-1}}$ as follows:
its vertices are $m$-diagonals and the arrows send the diagonal $(i,j)$ to $(i,j+m)$ and $(i+m,j)$ whenever
the image is also an $m$-diagonal. If $m=1,$ we omit the index $m$ and write $\Gamma_{A_{n-1}}$.
Furthermore, defining an automorphism $\tau_m$ as the map sending the $m$-diagonal $(i,j)$ to $(i-m,j-m)$ one makes the pair
$(\Gamma^m_{A_{n-1}}, \tau_m)$ into a (stable) translation quiver in the sense of C. Riedtmann.
As $\tau_m$ is defined on all vertices, the translation quiver $(\Gamma^m_{A_{n-1}}, \tau_m)$ is stable.
We refer to the examples in Section \ref{ex} for an illustration of this construction.

\begin{thm}\label{thmARdiag}
Let $Q$ be a quiver of Dynkin type $A_{n-1},$ and let $m\geq1.$ Then the Auslander-Reiten quiver of
$\mathcal{C}^m$ is isomorphic to $\Gamma^m_{A_{n-1}}$.
\end{thm}
Because of this theorem in the following we tacitly switch from one quiver to the other.

Now we recall the procedure of taking the $m$-th power of a given translation quiver.
\begin{de}
Let $(\Gamma, \tau)$ be a translation quiver. The quiver $(\Gamma)^m$ whose vertices are the same as the ones from $\Gamma$
and whose arrows are sectional paths of length $m$ is called the {\em$m$-th power of $\Gamma$}. A path 
$(x=x_0\rightarrow x_1\rightarrow...\rightarrow x_{m-1}\rightarrow x_m=y)$ in $\Gamma$ is said to be
sectional if $\tau x_{i+1}\neq x_{i-1}$ for $i=1,...,m-1$ (for which $\tau x_{i+1}$ is defined )
\end{de}
Observe that the pair $((\Gamma)^m,\tau^m)$ where $\tau^m:=\tau\circ...\circ\tau,$ composed $m$ times, is again
a translation quiver ( Theorem 6.1 in \cite{bm1}), and if $\Gamma$ is a stable translation quiver, then so is $((\Gamma)^m,\tau^m)$.
Furthermore, the next result (Proposition 7.1 in \cite{bm1}) shows how the quivers introduced so far are related to each other.
\begin{prop}\label{conne}
 $(\Gamma^m_{A_{n-1}},\tau_m)$ is a connected component of the $m$-th power of $(\Gamma_{A_{nm-1}},\tau).$
\end{prop}
Generally the $m$-th power of $(\Gamma_{A_{nm-1}},\tau)$ decomposes into several different
connected components. In the last section we will characterize all of them and we will see that apart from 
$(\Gamma^m_{A_{n-1}}, \tau_m)$ there may also occur other components which are isomorphic to the Auslander-Reiten
quiver of $\widetilde{m}$-cluster categories of type $A_n$, for suitable $\widetilde{m}\geq1$.

\section{Geometric description of the triangulated structure of $m$-cluster categories} 

In this section we explain how the triangulated structure of an $m$-cluster category of type $A_{n-1}$ can be comprehended geometrically, and
how it can be understood through the combinatorial tool of taking sectional paths of length $m$.
The first part is devoted to the study of
Auslander-Reiten (AR) triangles, also called almost split triangles, which have a nice geometric description as we will see. 
In the second part we will give the link between AR
triangles of $\mathcal{C}^m$ and the ones of $\mathcal{C}.$ Here we use that $m$-cluster category of type $A_{n-1}$ 
can be obtained from the cluster category of type $A_{nm-1}$
by taking the $m$-th power of the (translation) quiver
$(\Gamma_{A_{nm-1}},\tau)$. In the last part we describe how the triangles 
$A\stackrel{\mu}{\rightarrow} B\rightarrow C\rightarrow\Sigma A$ in $\mathcal{C}$
and $\mathcal{C}^m$
associated to an arbitrary morphism $\mu:A\rightarrow B$ are linked.
For this we will concretely determine the geometric characterization of the object $C$.

In the following we simply write $\mathcal{D}:=\mathcal{D}^b(\mathrm{mod }kQ)$, since 
the choice of $Q$ should be clear from the context.
Recall that the indecomposables of
$\mathcal{D}$ are isomorphic to the stalk complexes with indecomposable stalk (see Paragraph I.5. in \cite{Ha2}).
That is: complexes $X^{\bullet}$ for which there is an index $i_0$ such that $X^{i_0}\neq 0$ and $X^{i}=0$ for all $i\neq i_0.$
The object $X^{i_0}$ is then called stalk. This justifies the identification of $\mathrm{Ind}(\mathcal{D})$ with
$\mathrm{Ind}(\mathrm{mod }kQ)$ used in the following, where $\mathrm{Ind}(?)$ denotes
the full subcategory of indecomposables objects of the corresponding category.

\subsection{Auslander-Reiten triangles of $\mathcal{C}^m$.} For the geometric
description of the AR triangles it is convenient to introduce the following
and observe that almost split sequences are uniquely determined up to isomorphisms
by their ending or starting terms.

Notice that we always take the endpoints of diagonals modulo $N:=nm+2$.
\begin{lemma}\label{crossing}
The $m$-diagonals $(i,j)$ and $(i-m,j-m)$
cross inside $\Pi$ for all $(i,j)$.
\end{lemma}
\begin{proof}
A $m$-diagonal $(i,j)$ can be written as $(i,i+km+1)$ where we assume $i<j$ when writing $(i,j)$. 
From the definition of $\tau_m$ one deduces that the rotated copy $(i',j')$ of $(i,j)$ can be written as 
$(i-m,i+(k-1)m+1)$. Then, one of the following inequalities holds:
$$
i'<i<j'<j,\hspace{2mm}\textrm{ or   }\hspace{5mm} i < i'<j <j' ,
$$
depending on whether $j<j'$ or $j'<j$, and these inequalities define a crossing inside $\Pi$.
\end{proof}

\begin{de}
We call the set of $m$-diagonals $(i,j),(i,j-m), (i-m,j), \tau_m(i,j)=(i-m,j-m)$
written modulo $N$ the {\em framed set of $m$-diagonals of $(i,j)$}. 
The {\em frame} of a crossing $(i,j),(i-m,j-m)$ is given by the $m$-diagonals $(i,j-m),(i-m,j)$.
\end{de}

Notice that framed sets of $m$-diagonals arise from meshes in the translation quiver
\begin{equation*}  
\small
\xymatrix{       &(i-m,j)\ar[dr]  &                        \\
             \tau_m(i,j)\ar[rd]\ar[ru]\ar@{.}[rr] &&  (i,j)  \\
                     &(i,j-m)    \ar[ur]   &                           }
\end{equation*} 
defined by 4 or 3 vertices depending on where $(i,j)$ lies. Furthermore, the objects
corresponding to the tuples $(i-m,j)$ or $(i,j-m)$  which are not diagonals in $\Pi$ are viewed as zero.
In Figure \ref{fig:triangulations} two framed sets of $m$-diagonals are represented.

Next we point out a link which follows from the
definition of AR quivers and Theorem \ref{thmARdiag}, and which will often be used in the following.
\begin{rem} \label{Almsp}
The almost split sequence
$$ 0 \rightarrow M_{\tau_m(i,j)} \rightarrow  M_{(i-m,j)}\oplus M_{(i,j-m)} \rightarrow M_{(i,j)}\rightarrow 0 $$
ending at $M_{(i,j)}$ corresponds to the framed sets of $m$-diagonals of $(i,j)$.
\begin{figure}[h]
 \unitlength=1cm
   \begin{center}
      \psfragscanon
      \psfrag{testo0}{$(i-m,j)$}
      \psfrag{testo1}{$(i,j)$}
      \psfrag{testo2}{$\tau_m(i,j)$}
      \psfrag{testo3}{$(i,j-m)$}
      \psfrag{testo5}{$(i-m,j)$}
      \psfrag{testo6}{$(i,j)$}
      \psfrag{testo7}{$\tau_m(i,j)$}
\includegraphics[width=0.4 \textwidth]{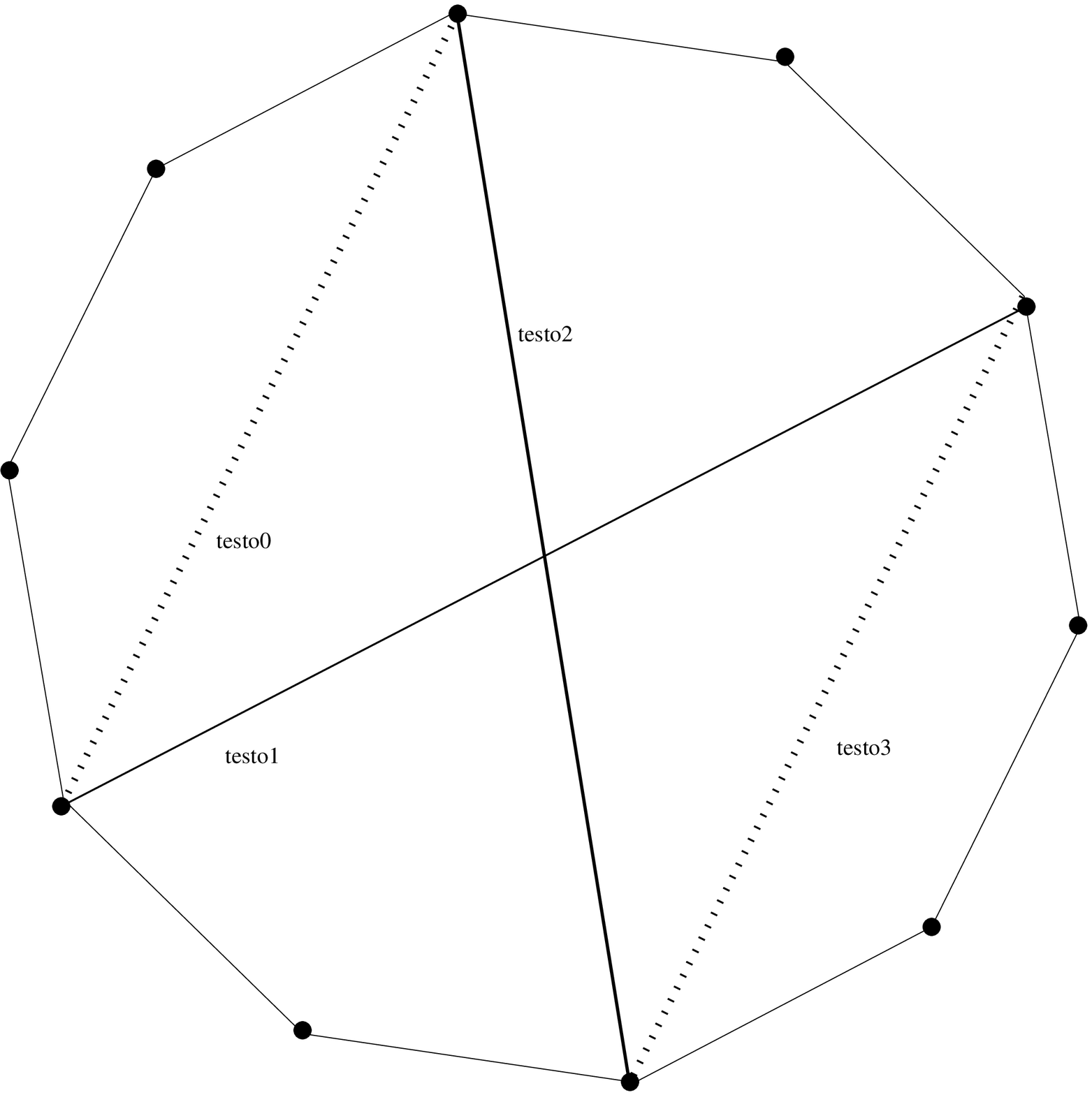} 
\includegraphics[width=0.4 \textwidth]{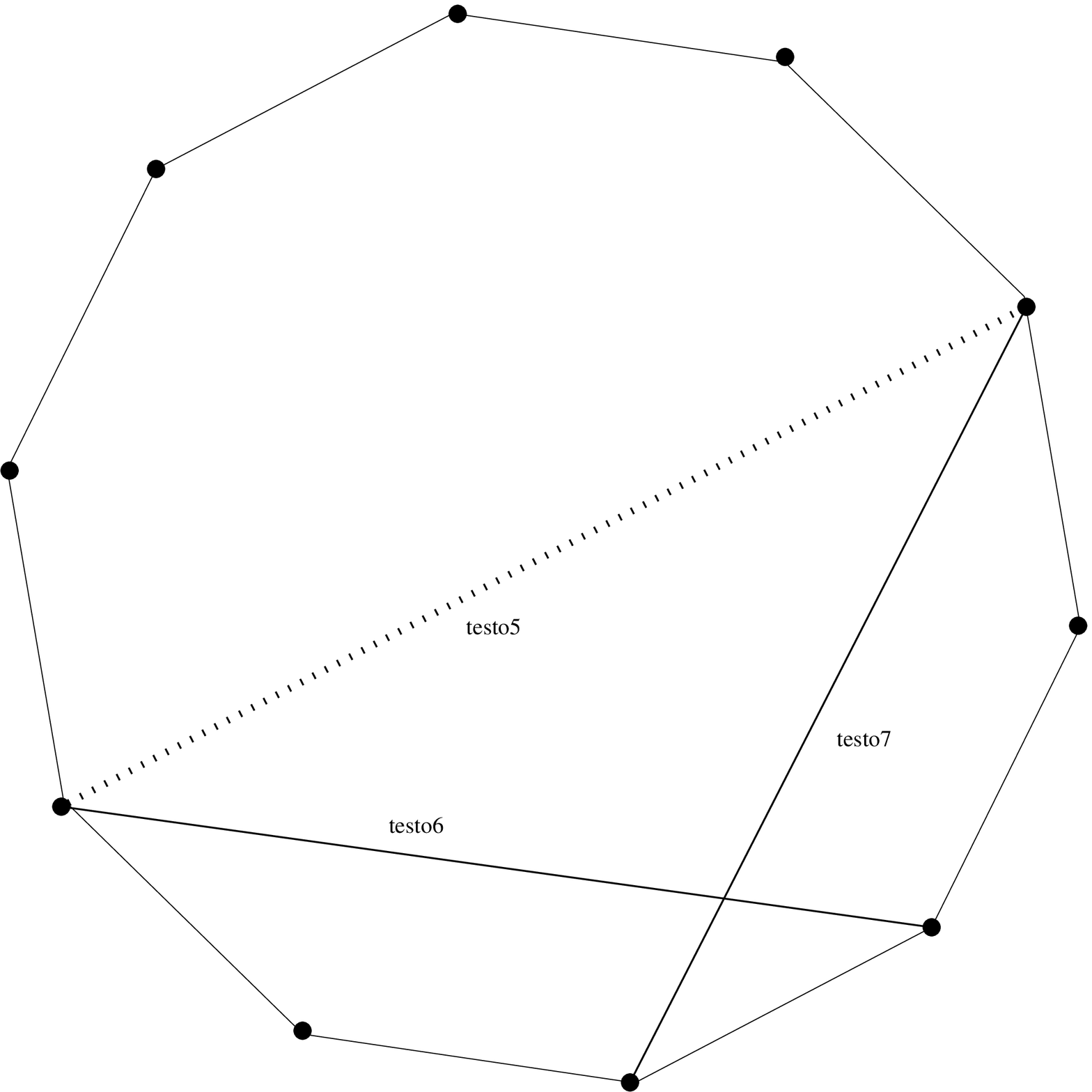}
   \end{center}
\caption{Framed set of $m$-diagonals }\label{fig:triangulations}
\end{figure} 

\end{rem}

For $1$-cluster categories the description of almost split  
sequences  in $\mathrm{mod }kQ$ for $Q$ of type $A_n$ is mentioned in the proof of theorem 5.1 in \cite{ccs}. For 
type $D_n,$ R. Schiffler gave the geometric interpretation of AR triangles in \cite{Sch}.\newline

After the previous observation the description of AR triangles of $\mathcal{C}^m$ also becomes clear.
Recall that $\pi:\mathcal{D}\rightarrow\mathcal{C}^m$ is the projection functor is also clear.

\begin{prop}\label{triang}
Let $M_{(i,j)}\in\mathrm{Ind}(\mathcal{C}^m)$. Then its AR triangle 
is 
$$
M_{\tau_m(i,j)} \rightarrow  M_{(i-m,j)}\oplus M_{(i,j-m)}\rightarrow M_{(i,j)}\stackrel{\pi(w)}{\rightarrow} \Sigma M_{\tau_{m}(i,j)},
$$
with $w$ different from zero.
\end{prop}
\begin{proof} We divide the proof into two cases depending on whether
the indecomposable object $M_{(i,j)}$ corresponds to a projective object of $\mathrm{mod}kQ$ or not.
Assume first it does not correspond to a projective. 
Then, since the category $\mathrm{mod} k Q$ is hereditary, the almost split sequence ending in $M_{(i,j)}$ described in Lemma \ref{Almsp} 
can be embedded into the following AR triangle of $\mathcal{D}:$
$$
M_{\tau_m(i,j)} \rightarrow  M_{(i-m,j)}\oplus M_{(i,j-m)} \rightarrow M_{(i,j)}\stackrel{w}{\rightarrow} \Sigma M_{\tau_{m}(i,j)},
$$ 
with $w\in\mathrm{Hom}_{\mathcal{D}}(M_{(i,j)} ,\Sigma M_{\tau_{m}(i,j)})=\mathrm{Ext}^1_{\mathrm{mod}kQ}(M_{(i,j)},M_{\tau_m(i,j)})$, 
see I.4.7 in \cite{Ha2}. Furthermore, as we saw in Lemma \ref{crossing} the diagonals $(i,j)$ and $ \tau_m(i,j)$, which
correspond to $M_{(i,j)}$ and $M_{\tau_m(i,j)}$, cross. Hence 
$$\mathrm{Ext}^1_{\mathrm{mod}kQ}(M_{(i,j)},M_{\tau_m(i,j)})\neq 0,$$
as it follows from Remark 5.3 in
\cite{ccs}. 
Applying the triangle functor $\pi:\mathcal{D}\rightarrow\mathcal{C}^m$ 
to it one gets the described triangle in $\mathcal{C}^m$ corresponding to the mesh
ending in the vertex $(i,j)$ of $\Gamma^m_{A_{n-1}}.$ 
For indecomposable objects of $\mathcal{C}^m$ which correspond to projective indecomposables 
of $\mathrm{mod}k Q$ one has that the
AR triangle, in this case also called connecting triangle, is the claimed one by Lemma 4.3 in \cite{Ha1}.
\end{proof}
In the following we will refer to the triangle of Proposition \ref{triang} as the {\em AR triangle of $M_{(i,j)}$}.
Notice that with the same arguments as the one
described in the proof of the previous lemma and taking $Q=A_{nm-1}$ one obtains the AR triangle 
of $M_{(i,j)}$ in $\mathcal{C}$ associated to the vertex $(i,j)$  of $(\Gamma_{A_{nm-1}},\tau).$

In geometric terms the previous result can be restated. 
Observe that the automorphism $\tau_m$ corresponds to an anticlockwise rotation of $\rho_m:=\frac{2m\pi}{ N}$
degrees inside $\Pi$ around the center.

\begin{cor}\label{mestri}
The AR triangle of ${M}_{(i,j)}$ in $\mathcal{C}^m$ 
$$
M_{\tau_m(i,j)}\rightarrow  M_{(i-m,j)}\oplus M_{(i,j-m)} \rightarrow M_{(i,j)}\stackrel{w}{\rightarrow} \Sigma M_{\tau_{m}(i,j)},
$$ 
is described by union of the framed set of $m$-diagonals of $(i,j)$ with an anticlockwise rotation of $(i,j)$ around the center
of $2\rho_m$ degrees.
\end{cor}
\begin{proof}
Follows from the two previous results and the fact that the shift 
functor $\Sigma$ in $\mathcal{C}^m$ is induced from the shift functor in $\mathcal{D}$ and satisfies $\Sigma\cong\tau_m.$
Thus it
corresponds to an anticlockwise rotation of $\rho_m$ degrees around the center
of $\Pi$, and the last term in the AR triangle  
of $M_{(i,j)}$ in $\mathcal{C}^m$
satisfies: $\Sigma M_{\tau_{m}(i,j)}= M_{2\tau_{m}(i,j)},$
i.e. is obtained from a rotation of $(i,j)$ of $2\rho_m$ degrees.
\end{proof}

We saw in Proposition \ref{conne} how the category $\mathcal{C}^m$ is related to $\mathcal{C}$
through the procedure of the $m$-th power of the translation quiver
$(\Gamma_{A_{nm-1}},\tau)$.  
Now, we will explain how the triangulated structures of these two categories
are related.

\begin{de} The {\em $m$-dilatation of a mesh} in $\Gamma_{A_{nm-1}}$ ending in $(i,j)$
is given by the mesh of $m$-diagonals $(i-m,j-m),(i-m,j),(i,j-m)$ and $(i,j)$, whenever defined.
\end{de}
Geometrically the $m$-dilatation enables us to pass 
from one mesh in $\Gamma$, or framed set of diagonals, to one in  $\Gamma^m$ in the following way.  One starts with the diagonal $(i,j)$ of $\Pi$ and
rotate it of $\rho_m$ degrees counterclockwise around the center of $\Pi$. This gives an $m$-diagonal $(i-m,j-m)$, which
crosses $(i,j)$ by Lemma \ref{crossing}. This crossing can be framed in a unique way by the two 
$m$-diagonals $(i-m,j)$ and $(i,j-m)$. The set of diagonals we obtain is then the $m$-dilatation of the mesh
of $(i,j)$.

\begin{prop} The $m$-dilatation of the mesh in $\Gamma_{A_{nm-1}}$ ending in $(i,j)$
maps the $AR$-triangle of $M_{(i,j)}$ in $\mathcal{C}$ to the one in $\mathcal{C}^m$.
\end{prop}
\begin{proof} The AR triangles of $M_{(i,j)}$ in $\mathcal{C}$, respectively in 
$\mathcal{C}^m$ were given in Corollary \ref{mestri}. The respective
framed sets of diagonals are $(i-1,j-1),(i,j-1),(i-1,j),(i,j)$ and
$(i-m,j-m),(i,j-m),(i-m,j),(i,j)$, which are linked by an $m$-dilatation.
\end{proof}

\subsection{Triangles arising from a morphism in $\mathcal{C}^m$.}
We saw that the indecomposable objects of $\mathcal{C}^m$ are isomorphic to certain indecomposable 
objects of $\mathcal{C}$,
and since  both categories are triangulated one has that
for a morphism $\mu$ from $M$ to $N$ there is a 
triangle $M\stackrel{\mu}{\rightarrow} N\rightarrow P\rightarrow \Sigma M $ 
in $\mathcal{C}$ and one in $\mathcal{C}^m$. In this section
we study how
those two triangles are related to each other. 
To answer this we first
recall the following definition, which will help understanding the morphisms in $\mathrm{Ind}(\mathcal{C}^m)$ as 
the latter is equivalent to the mesh category of $\Gamma^m_{A_{n-1}}$ (Theorem 5.6 in \cite{bm1}).
\begin{de}
The {\em mesh category} of a translation quiver $Q$ is the
factor category of the path category of $Q$ modulo the ideal generated by the mesh relations
$$
r_v:=\sum_{\alpha:u\rightarrow v}\alpha\cdot\sigma(\alpha),
$$
where the sum is over all arrows ending in $v$, and $v$ runs through the vertices of $Q$.
\end{de}
Geometrically the mesh relations are described in \cite{ccs} as follows.
First recall that two diagonals $\alpha$ and $\alpha'$ of $\Pi$ are related by a \textit{pivoting elementary 
move} if they share a vertex, called \textit{pivot},
and if there is a positive rotation around the pivot from $\alpha$ to $\alpha'$ inside $\Pi$. This rotation must be minimal.
A \textit{pivoting path} from $\alpha$ to $\alpha'$ is a sequence of pivoting elementary moves starting at $\alpha$ and ending at $\alpha'$.
If there are $m$ moves in such a path, we say that the path \textit{ has length} $m$. 
Now, for any two diagonals $\alpha$ and $\alpha'$ of $\Pi$ such that $\alpha$ is related to
$\alpha'$ by two consecutive pivoting elementary moves with distinct pivots, one defines the\textit{ mesh relations}: 
$P_{v'_2}P_{v_1}=P_{v'_1}P_{v_2},$ where $v_1,v_2$ are the vertices of the diagonal $\alpha$
and $v'_1,v'_2$ are the vertices of $\alpha',$ such that $P_{v'_1}P_{v_2}(\alpha)=\alpha'.$ In this definition
the border edges of $\Pi$ are allowed with the convention that the term corresponding
to such an edge in the mesh relation is replaced by zero.

As next we turn our attention to the morphisms in $\mathcal{C}^m,$ for this we let
$A:=\bigoplus_{1\leq i\leq m} M_i$ and $B:=\bigoplus_{1\leq j\leq n} N_j$ 
be objects in $\mathcal{C}^m$ with $M_i$ and $N_j$ indecomposables for all $i,j$'s.

Since $\mathcal{C}^m$ is an additive 
category one can identify an arbitrary $\mu\in\mathrm{Hom}_{\mathcal{C}^m}(A,B)$
with the $m \times n$ matrix  $(\mu_{i,j})_{1\leq i,j\leq m,n}$,
where $\mu_{i,j}:\pi_j \circ \mu \circ \iota_i\in\mathrm{Hom}_{\mathcal{C}^m}(M_i,N_j),$ with
$\pi_j:B\twoheadrightarrow N_j$,
and $\iota_i:M_i\hookrightarrow A$, as can be seen for example in Appendix A in \cite{ass}.

Furthermore, from the shape of the AR quiver of the $m$-cluster category it follows
that elements in $\mathrm{Hom}_{\mathcal{C}^m}(M_i,N_j)$
are compositions of irreducible morphisms, and the latter satisfy the following property.
\begin{lemma}\label{decomppath}
Let $M,N\in\mathrm{Ind}(\mathcal{C}^m)$. Then 
$\mu\in\mathrm{Irr}_{\mathcal{C}^m}(M,N)$ if and only if
$$
\mu=\nu_{m-1}\circ\dots\circ\nu_{0}
$$
for $I_i\in\mathcal{C},$ $0\leq i \leq m$, such that 
$I_0=M$, $I_{m}=N$ and $\tau I_{i+1}\ncong I_{i-1}$ for $0 < i < m$ and
$\nu_i\in \mathrm{Irr}_{\mathcal{C}}(I_i,I_{i+1})$ for $0\leq i \leq m-1$ 
\end{lemma}

\begin{proof} One implication is obvious, for the other one
let $\mu\in\mathrm{Irr}_{\mathcal{C}^m}(M,N)$, and consider the AR triangle 
in $\mathcal{C}^m$ 
$$M \stackrel{(\mu_1,\mu_2)}{\rightarrow} N_{1}\oplus N_{2}\rightarrow \tau_m^{-1} M
\rightarrow \Sigma M$$ where
$\mu_i\in\mathrm{Irr}_{\mathcal{C}^m}(M, N_i)$ for $i=1,2.$ 
From the shape of the AR quiver of $\mathcal{C}^m$ and since we assume both that 
$N$ is indecomposable and $\mu$ is irreducible, 
we deduce that either $N\cong N_1$ or $N\cong N_2.$ Otherwise
$\mu$ would be the zero morphism. Consequently, $\mu$ is either equal to $\mu_1$ or to $\mu_2$.
Furthermore, it follows
from Theorem \ref{thmARdiag}, Proposition \ref{conne} and the definition of $m$-th power of a translation quiver
that as morphism in $\mathcal{C}$, $\mu$ has to be of the form $\mu=\nu_{m_1}\circ\nu_{2_1}\circ\dots\circ\nu_{1_1}$ 
or $\mu=\nu_{m_2}\circ\nu_{2_2}\circ\dots\circ\nu_{1_2}$, here 
$\nu_{i_j}\in \mathrm{Irr}_{\mathcal{C}}(I_{i},I_{i_j})$ for $j=1,2$; $I_0:=M$ and $I_{m}:=N$.
This shows that $\mu$ corresponds to a sectional path of length $m$ in 
$\Gamma_{A_{nm-1}}$ starting at the vertex for $I_0$ and with $I_{i_j}$ corresponding to the vertices
along this path.

\end{proof}
In geometric terms the previous result says that if we  associate to the indecomposables
$M,N\in\mathcal{C}^m$ the $m$-diagonals $(i,j)$ and $(k,l)$ in $\Pi$
then $\mu\in\mathrm{Hom}_{\mathcal{C}^m}(M,N)$ is irreducible if and only if
it corresponds to a composition of $m$ elementary moves
around a common pivot, modulo the mesh relations.

To give a geometric meaning of the third term in the triangle associated to a morphism we want to study short exact
sequences. But since cluster categories are not abelian, not every morphism admits a kernel and cokernel, so we don't
have a direct access to short exact sequences. Nevertheless, the next result
tells us how to bridge this obstacle.

Remember from Proposition 1.6 in \cite{BMRRT} 
that for an $M\in\mathrm{Ind}(\mathcal{C})$ a representative of $\pi^{-1}(M)$
under the projection $\pi:\mathcal{D}\rightarrow\mathcal{C}$ is an object contained in the 
{\em fundamental domain} $\mathcal{S}$ for the action of
$F:=\tau^{-1}\Sigma$  on $\mathrm{Ind}(\mathcal{D})$. 
$\mathcal{S}$ is defined as the union of all indecomposable $k Q$-modules 
(viewed as stalk complexes concentrated in degree $0$) together
with the objects $\Sigma P_j,$ where the $P_j$'s are projective indecomposable $kQ$-modules.

\begin{lemma}\label{hommodkq}
For $M,N\in\mathrm{Ind}(\mathcal{C})$
$$
\mathrm{Hom}_{\mathcal{C}}(M,N)\cong\mathrm{Hom}_{\mathrm{mod}kQ}(M',N').
$$
 
\end{lemma}
Before passing to the proof observe that 
$\mathrm{Hom}_{\mathcal{D}}(M^\bullet,N^\bullet)\cong\mathrm{Hom}_{\mathrm{mod}kQ}(M,N)$, 
for $M^\bullet, N^\bullet$ indecomposable objects in $\mathcal{D}$ with indecomposable 
stalks $M, N$ concentrated in the same degree.
The above isomorphism is independent of the degree the stalks are concentrated in.

\begin{proof}[Proof of Lemma \ref{hommodkq}]
If the preimages $\widetilde{M},\widetilde{N}\in\mathcal{S}$ under $\pi$ of $M, N$ 
are both in $\mathrm{mod} kQ$ and 
$\widetilde{M}$ happens to be projective, the result follows from Proposition 1.7 in \cite{BMRRT}.
If this is not the case, we need to consider a $\tau^j$-shift of $\mathcal{S},$ for a certain $j\in\Z$.

In fact, assume that $\widetilde{M}\in\mathcal{S}$ is not a projective indecomposable $kQ$-module. Then 
the action of $F$ on the AR-quiver of $\mathrm{Ind}(\mathcal{D}),$ which as we pointed out earlier is
isomorphic to $\Z A_n$, implies that there
are $j,k,l\in\Z$ such that $F^j(\widetilde{M})\cong\Sigma^kP_l$, i.e. there is
a projective indecomposable $P_l$ concentrated in degree $k$ in the $F$-orbit of $\widetilde{M}$.
Call $\widetilde{M}':=\Sigma^kP_l$ and let $\widetilde{N}':=F^j(\widetilde{N})$ be the corresponding representative
of the $F$-orbit of $\widetilde{N}$ in the new fundamental domain 
$\mathcal{S}':=F^{j}(\mathcal{S}).$
Then we have to consider two cases.
First assume that  $\widetilde{N}'\ncong\Sigma^{k+1}P_i.$ 
Then $\widetilde{M}'$ and $\widetilde{N}'$ are both
concentrated in degree $k$ and $\widetilde{M}'$ is projective.
Thus 
\begin{align*}
\mathrm{Hom}_{\mathcal{C}}(M,N)&\cong\mathrm{Hom}_{\mathrm{mod}(kQ)}(\widetilde{M}',\widetilde{N}')
\end{align*}
again by Proposition 1.7 in \cite{BMRRT} together with the observation previous the proof.
Secondly, if $\widetilde{N}'\cong\Sigma^{k+1} P_i$
\begin{align*}
\mathrm{Hom}_{\mathcal{C}}(M,N)&\cong
\mathrm{Hom}_{\mathcal{D}}(F^{-1}\widetilde{M}',\widetilde{N}')\oplus\mathrm{Hom}_{\mathcal{D}}(\widetilde{M}',\widetilde{N}')\\
&\cong
\mathrm{Hom}_{\mathcal{D}}(\tau\Sigma^{k-1}P_l,\Sigma^{k+1}P_i)\oplus\mathrm{Hom}_{\mathcal{D}}(\Sigma^kP_l,\Sigma^{k+1}P_i)\\
&\cong
\mathrm{Hom}_{\mathcal{D}}(\tau P_l, \Sigma^2P_i)\oplus\mathrm{Hom}_{\mathcal{D}}(P_l,\Sigma P_i)\\
&\cong
\mathrm{Hom}_{\mathcal{D}}(\Sigma^{-1}I_{l'}, \Sigma^2P_i)\oplus\mathrm{Hom}_{\mathcal{D}}(P_l,\Sigma P_i)\\
&\cong
\mathrm{Hom}_{\mathcal{D}}(I_{l'},\Sigma^{3} P_i)\oplus\mathrm{Hom}_{\mathcal{D}}(P_l,\Sigma P_i)\\
&\cong
\mathrm{Ext}^3_{\mathrm{mod}kQ}(I_{l'}, P_i)\oplus\mathrm{Ext}^1_{\mathrm{mod}kQ}(P_l,P_i)\\
&=0\cong \mathrm{Hom}_{\mathrm{mod}kQ}(\widetilde{M}',\widetilde{N}')
\end{align*}
Here the first isomorphism follows from the fact that the category $\mathrm{mod}kQ$ is hereditary,
furthermore we used that $\tau P_l\cong \Sigma^{-1}I_{l'}$.
\end{proof} 
 
Our next goal is to determine the third object in the distinguished triangle 
$$
M\stackrel{\mu} {\rightarrow} N\rightarrow L\rightarrow\Sigma M \hspace{0.3cm}\textrm{in }\mathcal{C}
$$
into which the morphism $\mu$ embeds.

The notation in the next theorem is as follows: $M_{(i,j)}$,  $ M_{(k,l)}$ are the indecomposable 
objects associated to the diagonals $(i,j)$ and $(k,l)$ of $\Pi$ noticing that here $i<j$ and  $k\geq i$.

\begin{thm}\label{thmtri}
Let $\mu\in\mathrm{Hom}_{\mathcal{C}}(M_{(i,j)},M_{(k,l)})$,
then 
\begin{equation*} 
\mathrm{Cone}(\mu)\cong 
\left\{
\begin{array}{rl}
M_{(j-1,l)} & \textrm{if } \mu \textrm{ is injective },\\
\Sigma M_{(i,1+k)} & \textrm{if } \mu \textrm{ is surjective },\\
0 & \textrm{if } \mu \textrm{ is an isomorphism },\\
\Sigma M_{(i,1+k)}\oplus M_{(j-1,l)}& \textrm{otherwise }.
\end{array} \right.
\end{equation*}
\end{thm}

\begin{proof}
The strategy of the proof is the following. First lift the morphism $\mu$
to $\mathcal{D}$ which by Lemma \ref{hommodkq} can be done easily  in this case. Then complete
the morphism to a distinguished triangle in $\mathcal{D}$ and explicitly compute the cone of $\mu$ in $\mathcal{D}$.
At the end project the resulting triangle through the triangle functor $\pi:\mathcal{D}\rightarrow \mathcal{C}$ to 
obtain the distinguished triangle associated to $\mu$ in $\mathcal{C}$. By abuse of notation we write 
$\mathrm{Cone}(\mu):=\pi(\mathrm{Cone}(\mu))$.

After doing a $\tau^j$ shift of the fundamental domain $\mathcal{S}$  for a certain $j\in\Z$ as described in the proof
of Remark  \ref{hommodkq} we can assume that $i=1$. Furthermore, we assume $Q$ to be equally oriented. 
We proceed considering the different cases that occur for $\mu$.
First case: assume $M_{(k,l)}=M_{(1,l)}$ with $l\neq j$, so that $\mu$ is injective.
Then, consider the short exact sequence of indecomposable modules in the abelian category $\mathrm{mod}kQ$ 
$$
0\rightarrow M_{(1,j)}\stackrel{\mu}{ \hookrightarrow} M_{(1,l)}\stackrel{\nu}{\twoheadrightarrow}\mathrm{coker}(\mu)\cong M_{(1,l)}/ M_{(1,j)}\rightarrow 0
$$
which by the embedding of $\mathrm{mod}kQ$ into $\mathcal{D}$
can be associated to the distinguished triangle
$$
M_{(1,j)}\stackrel{\mu}{ \rightarrow} M_{(1,l)}\rightarrow  M_{(1,l)}/ M_{(1,j)}\stackrel{\delta}{\rightarrow} \Sigma  M_{(1,j)}
$$
where $\delta= (id,0)\circ\phi^{-1}$ and
$\phi:\mathrm{cone}(\mu)\rightarrow M_{(1,l)}/ M_{(1,j)}$ is the map of complexes: 
$$
\xymatrix{
\dots\ar[r]& M_{(1,j)}\ar[r]^{\mu}\ar[d]&\ar[d]M_{(1,l)}\ar[r]& 0 \ar[r]&\dots  \\
\dots\ar[r]&0\ar[r]& M_{(1,l)}/ M_{(1,j)}\ar[r]& 0\ar[r]&\dots }
$$
which is a quasi-isomorphism by Proposition 1.7.5 in \cite{KaS}. Notice that $\delta$ in $\mathcal{D}$ is
non zero since the short exact sequence does not split as $M_{(1,j)}$ is not an injective module (\cite{KaS}).
Moreover, if one identifies the vertices of the AR quiver of $\mathrm{mod}kQ$ with the positive roots of the root
system corresponding to $Q$ one easily sees that $M_{(1,l)}/ M_{(1,j)}\cong M_{(j-1,l)}$ in terms of objects associated to diagonals in $\Gamma_{A_{nm-1}}$.
One concludes considering that the previous triangle projects to the distinguished triangle of $\mathcal{C}$
$$
M_{(1,j)}\stackrel{\mu}{ \rightarrow} M_{(1,l)}\rightarrow \mathrm{Cone}(\mu)\rightarrow\Sigma  M_{(1,j)}.
$$ 
So that $\mathrm{Cone}(\mu)\cong M_{(j-1,l)}$ by the triangulated version of the Five Lemma
applied in $\mathcal{C}$.

Second case: assume that the target of $\mu$ is $M_{(k,j)}$ and $i< k \leq j-2$. So $\mu$ is surjective 
and one has the short exact sequence
$$
0\rightarrow\mathrm{ker}(\mu)\hookrightarrow M_{(1,j)}\stackrel{\mu}{\twoheadrightarrow} M_{(k,j)}\rightarrow 0.
$$
This can be embedded into the distinguished triangle
in $\mathcal{D}$:
$$
M_{(1,1+k)}\stackrel{u}{\rightarrow}  M_{(1,j)}\stackrel{\mu}{\rightarrow} M_{(k,j)}\stackrel{\nu}{\rightarrow} \Sigma M_{(1,1+k)}
$$
noticing that $\mathrm{ker}(\mu)\cong M_{(1,1+k)}$, as it follows from the shape 
of the AR quiver of $\mathrm{mod}kQ$. 
Since $M_{(k,j)}$ is never projective and $M_{(1,1+k)}$ is never injective, it follows that $\nu$
is non zero, as the short exact sequence does not split.
By the rotation axiom in a triangulated category $\mathcal{D}$ one has that the upper row in the next diagram is 
again a distinguished triangle 
$$
\xymatrix{
M_{(1,j)}\ar@^{=}[d]  \ar[r]^{\mu}& M_{(k,j)}\ar@^{=}[d]\ar[r]^{\nu}&
\Sigma M_{(1,1+k)}\ar@{.>}[d]^{\exists}\ar[r]^{-\Sigma u}&\Sigma M_{(1,j)}\ar@^{=}[d]\\
M_{(1,j)}\ar[r]^{\mu}& M_{(k,j)}\ar[r]&\mathrm{Cone}(\mu)\ar[r]&\Sigma M_{(1,j)}
}
$$
Hence as in the previous case applying $\pi$ one obtains that $\mathrm{Cone}(\mu)\cong \Sigma M_{(1,1+k)}$ in $\mathcal{C}$.

Third case: if $\mu$ is an isomorphism it follows from the previous cases that $\mathrm{Cone}(\mu)$ 
must be isomorphic to the zero object.

Last case: $\mu$ is neither surjective nor injective. In this 
situation it follows from the shape of the AR quiver of $\mathrm{mod}kQ$
that $\mu$
can be chosen to be the composition of  $f: M_{(1,j)}\hookrightarrow M_{(1,l)}, $ 
followed by $g: M_{(1,l)}\twoheadrightarrow M_{(k,l)}$.
Combining the triangles associated to these morphisms 
obtained in the same way as in the previous cases one gets:
$$
\xymatrix{
\mathrm{ coker}(f) \ar@/_ 0.7cm/[ddrr]_{\exists f'}  \ar[rd]^>{\bullet} &&  M_{(1,l)} \ar[ll] \ar@^{>>}[rd]^g && 
\Sigma \mathrm{ ker } (g)\ar[ll]_>\bullet\ar@/_ 0.7cm/[llll]_>\bullet  \\
& M_{(1,j)}\ar[rr]^{\mu}\ar@^{(->}[ru]^f&& M_{(k,l)}\ar[ru]\ar[dl]& \\
&& \mathrm{Cone}(\mu)\ar[ul]_>\bullet\ar@/_ 0.7cm/[uurr]_{\exists g'}&&
}
$$
Here $\mathrm{ coker}(f)\cong M_{(j-1,l)}$ and $\Sigma \mathrm{ ker } (g) \cong \Sigma M_{(1,1+k)}$.
By the octahedral axiom one obtains the distinguished triangle 
$$
M_{(j-1,l)}\rightarrow\mathrm{Cone}(\mu)\rightarrow \Sigma M_{(1,1+k)} \stackrel{h}{\rightarrow }\Sigma M_{(j-1,l)}
$$
with $h:\Sigma M_{(1,1+k)}\rightarrow \Sigma M_{(j-1,l)}$ 
the zero morphism. In fact
$$\mathrm{Hom}_{\mathcal{D}}(\Sigma M_{(1,1+k)},\Sigma M_{(j-1,l)})\cong \mathrm{Hom}_{\mathcal{D}}(M_{(1,1+k)},M_{(j-1,l)}) $$
and it can be checked that $M_{(j-1,l)}$ is outside of the support of $\mathrm{Hom}_{\mathcal{D}}(M_{(1,1+k)}, ?)$.

By the rotation axiom applied to the previous triangle one obtains the first row in the next diagram:
$$
\xymatrix{
M_{(1,1+k)}\ar@^{=}[d]  \ar[r]^{0}& M_{(j-1,l)}\ar@^{=}[d]\ar[r]&\mathrm{Cone}(\mu)\ar@^{>}[d]^{\exists}\ar[r]^{}&\Sigma M_{(1,1+k)}\ar@^{=}[d]\\
M_{(1,1+k)}\ar[r]^{0}& M_{(j-1,l)}\ar[r]& \Sigma M_{(1,1+k)}\oplus M_{(j-1,l)} \ar[r]&\Sigma M_{(1,1+k)}
}
$$
Here we used that $\mathrm{Cone}(0)\cong\Sigma M_{(1,1+k)}\oplus M_{(j-1,l)}$. The Five-Lemma
then implies that $\mathrm{Cone}(\mu)\cong \Sigma M_{(1,1+k)}\oplus M_{(j-1,l)}$ in $\mathcal{D}$. 
One concludes projecting the triangle to $\mathcal{C}$.
\end{proof}

In particular, it follows from the next result that taking the $m$-th power of the AR quiver
$\Gamma_{A_{nm-1}}$ 
of $\mathcal{C}$ doesn't give rise to new 
triangles between indecomposable objects.
\begin{cor} Let $M,N\in\mathrm{Ind}(\mathcal{C}^m)$ and $\mu\in \mathrm{Hom}_{\mathcal{C}^m}(M,N)$. 
Then $\mathrm{Cone}(\mu)\in\mathcal{C}^m$ is as described in Theorem \ref{thmtri}.
\end{cor}
\begin{proof} 
From Lemma \ref{decomppath} one deduces that $\mathcal{C}^m$ is a subcategory of $\mathcal{C}$, 
therefore it follows
from the Five-Lemma that the triangle associated to $\mu$
in $\mathcal{C}^m$ is isomorphic to the corresponding one in $\mathcal{C}$.
\end{proof}

\section{Components of $((\Gamma_{\textrm{A}_{mn-1}})^m,\tau^m)$ }

In Section 2 we saw that $(\Gamma^m_{A_{n-1}},\tau_m)$ is a connected component of the $m$-th power of $(\Gamma_{\textrm{A}_{mn-1}},\tau)$. 
So far it was not clear how to characterize all the remaining connected components arising when taking the $m$-th power
of the translation quiver $(\Gamma_{\textrm{A}_{mn-1}},\tau)$.
In fact, the answer depends on the parity of $m$. The
main difference is that when $m$ is even a further symmetry in the translation quiver
$(\Gamma_{\textrm{A}_{mn-1}},\tau)$ appears,
due to the symmetry of the polygon $\Pi$ around its central diagonals, i.e. those of the form $(i,i+\frac{N}{2})$. 
This symmetry increases the number of connected components arising.

This section is divided into two parts and in the first one
we present two results of the unpublished work of C.Ducrest, \cite{Du}, who was able to describe all
connected components for odd values of $m$. 
In the second part we focus on even $m$ and give a complete characterization of the connected components.
Our strategy is to study the action of the automorphisms $\tau^{-1}$ and $\Sigma$ on $\Z A_{n}$.
For this it is convenient to
declare the horizontal distance between two neighbored vertices of $\Z A_n$ to be one. 

\begin{rem} \label{rem} 
The action of $\tau^{-1}$ can be seen as moving
each vertex one unit to the right. The action of $\Sigma^m$ is to move each vertex
$m\frac{n+1}{2}$ units to the right followed by a reflection around the horizontal center line whenever $m$ is odd.
\end{rem}
From this it follows that $\Z A_n/\tau^{-1}\Sigma^m$ has the shape of a cylinder when $m$ is even, or of a M\"obius band if $m$ is odd.

\begin{de}
Let $(\Gamma,\tau)$ be $(\Gamma_{\textrm{A}_{mn-1}},\tau)$ or $(\Gamma^m_{A_{n-1}},\tau_m)$.
We call the {\em $i$-column} of $\Gamma$ the line 
in which all the diagonals, or $m$-diagonals, of the polygon $\Pi$ containing the vertex $i$ are listed.
\end{de}
Remark that $i$-columns come with two possible slopes as they can go in two different directions. 
Furthermore, from now on we write $(\Gamma, \tau)$ for $(\Gamma_{\textrm{A}_{mn-1}},\tau)$.

The next result  
characterizes the $\tau$-orbits containing the vertices of $\Gamma$ in terms
of their $\tau_m$-orbits. In particular: we will see that
if $m$ is odd, the two orbits always coincide. 
\begin{prop}\label{ncomponents}
The number of connected components of  $((\Gamma)^m,\tau^m)$
meeting any $\tau$-orbit of $\Gamma$ is:
\begin{enumerate}
 \item 1 for m odd,
 \item 1 or 2 for m even, and if it is 2 then the corresponding connected components are isomorphic.
\end{enumerate}
Furthermore, each connected component of $((\Gamma)^m,\tau^m)$ contains vertices of exactly one 
of the $\tau$-orbits of $(1,3),(1,4),\dots,(1,\lceil\frac{m-1}{2}\rceil+2)$ or $(1,m) $. 
\end{prop}
\begin{proof}
This follows directly from Lemma 4.2.19 and 4.2.20 in \cite{Du}.
\end{proof}
Observe that the vertices in the  $\tau$-orbit of $(1,km)$, $k\geq 1,$ give rise to the connected
component of the $m$-th power of $\Gamma$ isomorphic to the AR quiver of $\mathcal{C}^m$, c.f. Theorem \ref{thmARdiag}.
The next result describes explicitly the connected components into which
the $m$-th power decomposes when $m$ is odd. An illustration of this result
is presented in Section \ref{ex}.
\begin{thm}\cite[Thm: 4.2.24]{Du}  \label{Co} For all $m\geq 1$, there exists a  $t$ such that
$$(\Gamma)^m=\Gamma^m_{\textrm{A}_{n-1}}\coprod \bigcup_{i=1}^{t}\Gamma_i,$$
is the decomposition of $((\Gamma)^m,\tau^m)$ into connected components, 
and each $\Gamma_i$ is isomorphic to the Auslander-Reiten quiver of the orbit category of $\mathcal{D}^b(k A_n)$ under an auto-equivalence of the form
$\tau^{-s}\circ\Sigma^r$, for some $r,s$ with $s<n.$ \newline
If $m$ is odd then $t=\frac{m-1}{2},$ $r=m+1$ and $s=\frac{m-1}{2}n-\frac{m-3}{2},$
i.e.
$$(\Gamma)^m=\Gamma^m_{\textrm{A}_{n-1}}\coprod \bigcup_{i=1}^{\frac{m-1}{2}} \Z A_n/(\tau^{-\frac{m-1}{2}n+\frac{m-3}{2}}\circ\Sigma^{m+1}).$$
\end{thm}
The explicit characterization of the auto-equivalence in the previous result yields the following:
\begin{cor}
Let $m$ be odd, then the connected components $\Gamma_i$ have all a cylindrical shape, whereas $\Gamma^m_{A_{n-1}}$ lies on a M\"obius band.
\end{cor}
\begin{proof}
Observe that the shift functor $\Sigma$ in the auto-equivalence: $\tau^{-\frac{m-1}{2}n+\frac{m-3}{2}}\circ\Sigma^{m+1}$ is applied an even number of times.
\end{proof}
Later on we will see that when $m$ is even this is no longer always true. In fact, a number of orbit categories
arising through $(\Gamma)^m$ will have a M\"{o}bius band as their Auslander-Reiten quiver.

\begin{prop}\label{pro1}
Let $\Gamma_i$ be as in Theorem \ref{Co} and $m$ odd. 
If $(n+1)$ divides $2(nm+1)$, then $\Gamma_i\cong \Gamma^u_{A_n}$ with
$u:=\frac{2(nm+1)}{n+1}$.
\end{prop}
\begin{proof} By Proposition 1.3 in \cite{BMRRT} we know that
the Auslander-Reiten quiver of the $u$-cluster category of type ${A_n}$ is isomorphic to $\Z A_n/\tau^{-1}\circ\Sigma^u.$
Set $\varphi:=\tau^{-s}\circ \Sigma^{m+1}$ where $s=\frac{m-1}{2}n-\frac{m-3}{2}$ by Theorem \ref{Co} 
and $\varphi':=\tau^{-1}\circ\Sigma^u$. Then compare
their actions on $\Z A_{n}$ which is isomorphic to the 
Auslander-Reiten quiver of
$\mathcal{D}^b(k A_{n})$ by \cite{Ha2}.
For $u$ satisfying the assumptions we have that the vertices of $\Z A_{n}$ have the same orbit under the
action of the two auto-equivalences $\varphi$ and $\varphi'$. In fact, each vertex is moved by
$(m+1)\frac{n+1}{2}+(\frac{m-1}{2}n-\frac{m-3}{2})=u\frac{n+1}{2}+1$ units to the right. Furthermore,  
since $u$ is even, $\Gamma_i\cong \Z A_n /\varphi$ and $\Z A_n /\varphi'$
are both cylindrically shaped. Thus $\Z Q/\varphi\cong\Z Q/\varphi'.$
\end{proof}

\begin{ex}
Let $m=5$ and $n=3,$ and consider the connected components of $(\Gamma_{A_{14}})^5$. There are three of them: 
one is $\Gamma^5_{A_2}$, and the other two are isomorphic to $\Gamma^8_{A_3}$, 
which corresponds to the Auslander-Reiten quiver of the $8$-cluster category of type $A_3$.
\end{ex}

\subsection{The case where $m$ is even}

From the assumption on $m$ it follows that the $N$-gon $\Pi$ has an even number of sides. 
The most important difference to the case of odd $m$ is that the polygon
contains central diagonals $(i,i+\frac{N}{2}).$ Let us consider the diagonals incident
with the vertex $i$. Apart from the central diagonal 
$(i,i+\frac{N}{2})$ they come in pairs $(i,j)$ and $(i,N+1+i-j)$, symmetric about $(i,i+\frac{N}{2})$.
We will use  $j^-:=N+2-j$ and write $(i,j^-)$ for the diagonal which is the mirror
of $(i,j)$.
Observe that the $\tau$-orbit containing the central diagonals forms the middle row of $(\Gamma,\tau)$.

$$
\small
\xymatrix@-7,5mm{
 & & & & & &
 \diagramnode{$(1,{j_3}^-)$}\hspace{2mm}\ar@{--}[rr]\ar[rd] && 
 \diagramnode{$\bullet$ }\ar[rd]\\
 & & & & &
 \diagramnode{$(1,{j_4}^-)$ }\hspace{2mm}\ar@{--}[rr]\ar[rd]\ar[ru]&& 
 \diagramnode{$\bullet$} \ar@{--}[rr]\ar[rd]\ar[ru] && 
 \diagramnode{$\bullet$}\ar[rd]\\
 & & & & 
 \diagramnode{$\iddots$}\hspace{2mm}\ar@{--}[rr]\ar[rd]\ar[ru] &&
 \diagramnode{$\iddots$}\ar@{--}[rr]\ar[rd]\ar[ru] && 
 \diagramnode{$\iddots$}\ar@{--}[rr]\ar[rd]\ar[ru] && 
 \diagramnode{$\ddots$}\ar[rd]\\
 & & &
\diagramnode{$(1,\frac{N+2}{2})$}\ar@{--}[rr]\ar[rd]\ar[ru] && 
 \diagramnode{$\bullet$}\ar@{--}[rr]\ar[rd]\ar[ru] && 
 \diagramnode{$\bullet$}\ar@{--}[rr]\ar[rd]\ar[ru] &&
 \diagramnode{$\bullet$}\ar@{--}[rr]\ar[rd]\ar[ru] && 
 \diagramnode{$\bullet$}\ar[rd]\\
 &&
 \diagramnode{$\iddots$}\hspace{2mm}\ar@{--}[rr]\ar[rd]\ar[ru] && 
 \diagramnode{$\iddots$}\ar@{--}[rr]\ar[rd]\ar[ru] && 
 \diagramnode{$\iddots$}\ar@{--}[rr]\ar[rd]\ar[ru] &&
 \diagramnode{$\iddots$}\ar@{--}[rr]\ar[rd]\ar[ru] && 
 \diagramnode{$\ddots$}\ar@{--}[rr]\ar[rd]\ar[ru] &&
 \diagramnode{$\ddots$}\ar[rd]\\
 &
 \diagramnode{$(1,j_4)$}\hspace{2mm}\ar@{--}[rr]\ar[ru]\ar[rd] && 
 \diagramnode{$\bullet$}\ar@{--}[rr]\ar[ru]\ar[rd] && 
 \diagramnode{$\bullet$}\ar@{--}[rr]\ar[ru]\ar[rd] && 
 \diagramnode{$\bullet$}\ar@{--}[rr]\ar[ru]\ar[rd] && 
 \diagramnode{$\bullet$} \ar@{--}[rr]\ar[ru]\ar[rd]&&
 \diagramnode{$\bullet$}\ar@{--}[rr]\ar[ru]\ar[rd] && 
 \diagramnode{$\bullet$}\ar[rd]\\
 \diagramnode{$(1,j_3)$}\hspace{2mm}\ar@{--}[rr]\ar[ru] && 
 \diagramnode{$\bullet$}\ar@{--}[rr]\ar[ru] && 
 \diagramnode{$\bullet$}\ar@{--}[rr]\ar[ru] && 
 \diagramnode{$\bullet$}\ar@{--}[rr]\ar[ru] && 
 \diagramnode{$\bullet$}\ar@{--}[rr]\ar[ru] &&
 \diagramnode{$\bullet$}\ar@{--}[rr]\ar[ru] &&
 \diagramnode{$\bullet$}\ar@{--}[rr]\ar[ru] &&
 \diagramnode{$\bullet$}
}
$$
\begin{rem}\label{tau orbit}
The $\tau$-orbit of $(1,j)$ also contains $(1,j^-)$ and
the cardinality of it is $N$ unless $j=j^-$. If so, the  cardinality is
$\frac{N}{2}.$
\end{rem}

Furthermore, it follows from the definition of $\tau^m$ that
by distinguishing the endpoints of a given diagonal $(i,j)$ according to their parity one has:
\begin{itemize}
 \item $\tau^m (\mathrm{even},\mathrm{even})=(\mathrm{even},\mathrm{even})$
 \item $\tau^m (\mathrm{odd},\mathrm{odd})= (\mathrm{odd}, \mathrm{odd})$
 \item $\tau^m (\mathrm{even},\mathrm{odd}),\tau^m (\mathrm{odd},\mathrm{even}) 
\subset\left\{(\mathrm{odd}, \mathrm{even}),(\mathrm{even}, \mathrm{odd})\right\}.$
\end{itemize}
We will refer to these properties as {\em parity configurations}.

The next results sharpen Proposition \ref{ncomponents}.
\begin{lemma}\label{jodd} Let $(i,j)$ be a vertex of $\Pi$. Assume that
$m$ and $\left|i-j\right|$ are even.
Then the $\tau$-orbit of $(i,j)$ 
meets two components of $((\Gamma)^m,\tau^m)$.
\end{lemma}
\begin{proof} Without lost of generality we can assume $i=1.$

Since $j$ is odd it follows from the parity configurations 
that all the elements of
the $\tau^m$-orbit containing $(i,j)$ are diagonals 
whose end points are both odd.
Hence, diagonals with vertices of the form
$(\mathrm{even},\mathrm{even})$ never belong to the $\tau^m$-orbit of $(1,j)$. Therefore, 
there is strictly more than one component in the $\tau$-orbit of $\Gamma$ through $(1,j)$. 
From Proposition \ref{ncomponents} it follows that
there must be exactly two.
\end{proof}
\begin{cor} \label{cardtau}
Let $m$ and $\left|i-j\right|$ be even, then the 
cardinality of the $\tau^m$-orbit containing $(i,j)$ is $\frac{N}{2}$.
\end{cor}
\begin{proof} Follows from Lemma \ref{jodd} and Remark \ref{tau orbit}.
\end{proof}

It remains to consider $\tau$-orbits of vertices $(i,j)$ of $(\Gamma,\tau)$
for odd $\left|i-j\right|$. In the proof of the next lemma 
it becomes clear how the symmetry of the polygon interferes with the 
connected components.

\begin{lemma} \label{symm}
Let $m$ be even and $\left|i-j\right|$ be odd. Then, the $\tau$-orbit of $(i,j)$ meets only one connected component of $((\Gamma)^m,\tau^m)$
if and only if
$(i,j)$ and $(i,j^-)$ belong to the same connected component.
\end{lemma}

\begin{proof} As before w.l.o.g. let $i=1$ and
observe that the $\tau^m$-orbit containing $(1,j)$ and the $\tau^m$-orbit of $(1,j^-)$ are disjoint unless
$j=j^-,$ i.e. $j=\frac{N+2}{2}$.
In fact, modulo $N$ the $\tau^m$-orbit of $(1,j)$ consists of diagonals
of the form 
$(1-rm,j-rm)$.
In particular, all its elements have parity configuration: $(\mathrm{odd},\mathrm{even}),$ or $(\mathrm{even},\mathrm{odd})$.
Therefore, we only have to show that the $\tau^m$-orbit of $(1,j)$
does not contain $(2,j+1)$ exactly when $(1,j^{-})\not\in\tau^m(1,j)$.
Assume instead that there is a $k_0\in\N$ such that both $1-k_0m\equiv j+1$ and $j-k_0m\equiv2$ modulo $N.$ 
Then $N$ would divide their difference: $|(j+1-2)-(1-k_0m-j+k_0m) |=|2j-2|.$ But this 
can only happen if $(1,j)$ is a central diagonal, i.e $j=j^-$, which is equivalent to $j=\frac{N+2}{2}$.
Thus, we showed that 
$$(2,j+1)\in\tau^m(1,j)\Longleftrightarrow (1,j^-)\in \tau^m(1,j).$$

To conclude recall that the
$\tau$-orbit of $(1,j)$ coincides with the one of $(1,j^-)$. 
\end{proof}
\begin{cor}\label{card} 
Let $m$ be even and $\left|i-j\right|$ be odd, then
the cardinality of the $\tau^m$-orbit of $(i,j)$ is $\frac{N}{2}$.
\end{cor}
\begin{proof} In fact, when $\left|i-j\right|$ is odd 
the $\tau^m$-orbit of $(i,j)$ and the $\tau^m$-orbit of $(i,j^-)$ 
are disjoint if $j\neq j^-$, as we observed in the proof of Lemma \ref{symm}.
These two sets together give the $\tau$ orbit of $(i,j)$. 
If $j=j^-$ then the $\tau$ and $\tau^m$
orbit coincide. Hence, we conclude with Remark \ref{tau orbit}.
\end{proof}

From now one we denote by $\Gamma_{(i,j)}$ the connected component of $((\Gamma)^m,\tau^m)$
containing the diagonal $(i,j)$. 
\begin{cor}\label{mtwo} For $m=2$, $((\Gamma)^2,\tau^2)$ splits into three connected components
$$
(\Gamma)^2 = \Gamma^2_{A_{n-1}}\cup \Gamma_{(1,3)}\cup \Gamma_{(2,4)}.
$$
\end{cor}
\begin{proof} For $m=2,$ it is clear that for every even $j$, 
$(1,j)$ and  $(1,j^-)$ always belong to the same connected component.
Hence, by Lemma \ref{symm} 
the $\tau$-orbit of $(1,2k), k\in\N,$ meets one connected component of 
$((\Gamma)^m,\tau^m)$ and the $\tau$-orbit of $(1,2k+1), k \in\N$ meets two by Lemma \ref{jodd}.
\end{proof}
Observe that with Proposition \ref{ncomponents} we were only able to conclude that the
second power of $\Gamma$ consists of two or three connected components.

Next we give a criterion to determine whether the mirror diagonal $(i,j^-)$ of $(i,j)$, for $j\neq j^-$, 
is or is not part of the connected component containing $(i,j)$.
\begin{lemma}\label{mdivides}
Let $m$ be even, then $(i,j^-)$ and $(i,j)$ belong to the same
connected component of $((\Gamma)^m,\tau^m)$ if and only if $m$ divides $|2(2-j)|$.
\end{lemma}
\begin{proof} Let $i=1$. The previous Lemma \ref{symm} the $\tau$-orbit of $(1,j)$ meets 
one connected component
if and only if the vertices $(1,j)$ and $(1,j^-),$ where $j^-=N+2-j$, 
are part of the same connected
component of $((\Gamma)^m,\tau^m)$. This is equivalent to saying that
there is a $r_0\in\N$ such that $(n-r_0)m=4-2j.$
\end{proof}

\begin{prop}\label{oneortwo} For $m >2$ even and $|i-j|$ odd
the $\tau$-orbit of $(i,j)$ always meets two connected components 
except if $|i-j|=\frac{m}{2}+1$ then it only meets one connected component.
\end{prop}
\begin{proof} Follows from Lemmas \ref{jodd}, \ref{symm} and \ref{mdivides}.
In fact, $m \big| |2(2-j)|$ if and only if $j=\frac{m+4}{2}$, since for all other values of
$j$ one has that $m>|2(2-j)|.$
\end{proof}

Combining the previous results we are now able to state the main result of this section, which
completes the description of the connected components 
of the $m$-th power of $\Gamma$.

\begin{thm}\label{pro2}
Let $m$  be even. Then  $((\Gamma)^m,\tau^m)$ decomposes into $m$ or $m+1$ connected components. In fact
\begin{align*}
(\Gamma)^m=\Gamma^m_{A_{n-1}}\cup\bigcup_{i=1}^{\frac{m}{2}-1} \bigg( \Gamma_{(1,i+2)}\cup\Gamma_{(2,i+3)}\bigg)\cup\Gamma_{(1,\frac{m}{2}+2)}
 ,\hspace{1cm}\textrm{if} \hspace{0,3cm} \frac{m}{2}\hspace{0,3cm}\textrm{ is even};\\
(\Gamma)^m=\Gamma^m_{A_{n-1}}\cup\bigcup_{i=1}^{\frac{m}{2}-1}\bigg(\Gamma_{(1,i+2)}\cup\Gamma_{(2,i+3)}\bigg)\cup \bigcup_{k=1}^2\Gamma_{(k,\frac{m}{2}+1+k)}, \hspace{1cm}\textrm{if}\hspace{0,3cm} \frac{m}{2}\hspace{0,3cm}\textrm{ is odd }.
\end{align*}
Furthermore, all the connected components aside from $\Gamma^m_{A_{n-1}}$ are isomorphic to $\Z A_n/{\tau^{-\frac{n(m-2)}{2}}\Sigma^2}$ 
and cylindrically shaped except $\Gamma_{(1,\frac{m}{2}+2)}$ and $\Gamma_{(2,\frac{m}{2}+3)}$ 
which are isomorphic to $\Z A_n/{\tau^{-\frac{n(m-2)}{4}}\Sigma}$ and have the shape of a M\"{o}bius band.
\end{thm}

\begin{proof} To begin recall that by
Proposition \ref{ncomponents} we know that each connected component different from $\Gamma^m_{A_{n-1}}$ has a vertex in exactly one 
$\tau$-orbit of the form $\tau(1,j+2),$ for 
$j\in\{1,...,\left\lceil \frac{m-1}{2}\right\rceil=\frac{m}{2}\}$. By
Lemma \ref{jodd} and Proposition \ref{oneortwo} we know that all the $\tau$-orbits for $j<\frac{m}{2}$ 
meet two isomorphic connected
components of the $m$-th power of $\Gamma$, and that  $\{ \tau^k(1,\frac{m}{2}+2) \arrowvert k \geq 1 \}$ 
meets one or two depending on the parity of $j$. Hence the number of connected components in the decomposition is clear.

To concretely describe the automorphisms that give rise to the different connected components
set $j_0=\frac{m}{2}+2$, and as the shape of the connected components
depends on the parity of $j_0$ one proceeds by cases.
 
First assume  $j_0$ is odd. By Lemma \ref{jodd} we know that the $\tau$-orbit 
of $(1,j_0)$ in $\Gamma$
meets two connected components which are isomorphic by Proposition \ref{ncomponents}.
We now show that they have the shape of a M\"obius band.
In fact, by Lemma \ref{mdivides} and by the definition of the $m$-th power of $\Gamma$ one has that 
$(1,j_0)$ and $(1,j_0^-)$ belong to the $1$-column of $\Gamma_{(1,j_0)}$. Moreover, since the
$\tau^m$-orbit of $(1,j_0)$ and $(1,j_0^-)$ coincide, each connected
component is a M\"obius band. 

If $j_0$ is even, we know by Proposition \ref{oneortwo} that the $\tau$-orbit 
of $(1,j_0)$ in $\Gamma$ meets only one connected component. Furthermore,
by the same argument as in the previous case one has that $(1,j_0)$ and $(1,j_0^-)$ are 
vertices of the 1-column of $\Gamma_{(1,j_0)}$. But 
from the proof of Lemma \ref{symm} it is clear that the $\tau^m$-orbit of $(1,j_0)$ does not coincide
with the $\tau^m$- orbit of $(1,j_0^-)$. Hence,
$(1,j_0)$ and $(1,j_0^-)$ cannot be identified, 
and therefore the shape of the component
is the one of a cylinder.

Now, we only have to check what happens in the $\tau$-orbits of $(1,j)$ if $j<\frac{m}{2}+2.$
In this case one deduces from Lemma \ref{mdivides} that the mirror diagonal $(1,j^-)$ 
is never a vertex of the 1-column of $\Gamma_{(1,j)}$. Thus, as before we conclude
that the components have the shape of a cylinder.

To conclude we only need to determine the action of the auto-equivalences in terms of powers of the AR translate 
$\tau$ and the shift $\Sigma$. To do so we first observe that such an expression
is not unique. For our purpose it is convenient to choose the one with a minimal power of $\Sigma$.
And to determine the appropriate powers of $\tau$
one observes that by knowing $\Gamma$ one easily determines the
length of the cylindrically shaped connected components. In fact it is $\frac{N}{2}$ by
Corollaries \ref{cardtau} and \ref{card}.

For the M\"{o}bius bands one observes that the $\tau^m$-orbits have $\frac{N}{4}$
elements (except for the $\tau^m$-orbit of the middle row, in the case were $n$ is odd).
This follows from the fact that the union of the vertices of the two M\"{o}bius bands together 
form a cylindrically shaped component
of length $\frac{N}{2}$ and the $\tau^m$-orbits of size $\frac{N}{2}$ meet two $\tau^m$-orbits
of such a cylinder.

To express these lengths in terms of $\Sigma$
and $\tau$ one uses Remark \ref{rem}.
\end{proof}

One notices that all the connected components of the $m$-th power of $(\Gamma,\tau)$ 
can be viewed as orbit categories which are triangulated.
In fact, if one denotes the induced auto-equivalences
$\tau^{-\frac{n(m-2)}{2}}\Sigma^2$, $\tau^{-\frac{n(m-2)}{4}}\Sigma$ and $\tau^{-\frac{n(m-1)+(m-3)}{2}}\Sigma^{m+1}$
on $\mathcal{D}$
by $G_1,G_2$ and respectively by $G_3$, then $G_1,G_2$ and $G_3$ are triangle functors
satisfying the conditions of Theorem 1 in \cite{K1}.

Hence, it follows that the projection functors $\pi_s: \mathcal{D}\rightarrow \mathcal{D}/G_i$
are triangle functors for $i=1,2,3.$ Thus, one has

\begin{cor}
For all $m\in\N,$ the connected components of $(\Gamma)^m$ are
isomorphic to AR quivers of triangulated orbit categories.
\end{cor}

Furthermore, also for even values of $m$  it happens that some 
of the connected components arising while taking the $m$-power of $(\Gamma,\tau)$ are isomorphic to
AR quivers of $u$-cluster categories
for appropriate values of $u$. Compare with Lemma \ref{pro1}.
In fact according to the different cases treated in Theorem \ref{pro2} one can deduce the following.
\begin{prop} \label{CC}
Let $m$ be even. 
\begin{enumerate}
\item  The cylindrically shaped components $\Gamma_{(i,j)}$ of $(\Gamma)^m$
       are isomorphic to the AR quiver of an $u$-cluster category whenever $u:=\frac{nm}{n+1}$ is even.
\item  The other $\Gamma_{(i,j)}$'s are isomorphic to the AR quiver of
       an $u$-cluster category whenever $u:=\frac{nm-2}{2(n+1)}$ is odd. 
\end{enumerate}

\end{prop}
\begin{proof}
 The proof is analogous as the one of Proposition \ref{pro1} and therefore omitted.
\end{proof}

One observes that the results of Section 3 apply to all orbit categories
arising from the connected components of the $m$-th power of $\Gamma$.
Hence they provide a geometric understanding of the triangulated structure
of these categories.

\section{Application}\label{ex}
In this section we illustrate on two examples the results  presented in Section 4.
In particular we illustrate how the shape of the connected components described in
Theorem \ref{pro2} vary.
\begin{ex}\label{Ex}
Let $m=2$  and $n=4$. Consider the quiver $\Gamma_{A_7}$ associated to the diagonals of a decagon:
\[
\small
\xymatrix@-7,5mm{
 & & & & & &
 \diagramnode{19}\ar@{--}[rr]\ar[rd] && 
 \diagramnode{2,10} \ar@{--}[rr]\ar[rd] && 
 \diagramnode{13} \ar[rd]\\
 & & & & &
 \diagramnode{18}\ar@{--}[rr]\ar[rd]\ar[ru]&& 
 \diagramnode{29} \ar@{--}[rr]\ar[rd]\ar[ru] && 
 \diagramnode{3,10}\ar@{--}[rr]\ar[rd]\ar[ru] && 
 \diagramnode{14} \ar[rd]\\
 & & & & 
 \diagramnode{17}\ar@{--}[rr]\ar[rd]\ar[ru] &&
 \diagramnode{28}\ar@{--}[rr]\ar[rd]\ar[ru] && 
 \diagramnode{39}\ar@{--}[rr]\ar[rd]\ar[ru] && 
 \diagramnode{4,10}\ar@{--}[rr]\ar[rd]\ar[ru] && 
 \diagramnode{15} \ar[rd]\\
 & & &
 \diagramnode{16}\ar@{--}[rr]\ar[rd]\ar[ru] && 
 \diagramnode{27}\ar@{--}[rr]\ar[rd]\ar[ru] && 
 \diagramnode{38}\ar@{--}[rr]\ar[rd]\ar[ru] &&
 \diagramnode{49}\ar@{--}[rr]\ar[rd]\ar[ru] && 
 \diagramnode{5,10}\ar@{--}[rr]\ar[rd]\ar[ru] &&
 \diagramnode{16}\ar[rd] \\
 &&
 \diagramnode{15}\ar@{--}[rr]\ar[rd]\ar[ru] && 
 \diagramnode{26}\ar@{--}[rr]\ar[rd]\ar[ru] && 
 \diagramnode{37}\ar@{--}[rr]\ar[rd]\ar[ru] &&
 \diagramnode{48}\ar@{--}[rr]\ar[rd]\ar[ru] && 
 \diagramnode{59}\ar@{--}[rr]\ar[rd]\ar[ru] &&
 \diagramnode{6,10}\ar@{--}[rr]\ar[rd]\ar[ru]&& 
 \diagramnode{17} \ar[rd]\\
 &
 \diagramnode{14}\ar@{--}[rr]\ar[ru]\ar[rd] && 
 \diagramnode{25}\ar@{--}[rr]\ar[ru]\ar[rd] && 
 \diagramnode{36}\ar@{--}[rr]\ar[ru]\ar[rd] && 
 \diagramnode{47}\ar@{--}[rr]\ar[ru]\ar[rd] && 
 \diagramnode{58} \ar@{--}[rr]\ar[ru]\ar[rd]&&
 \diagramnode{69}\ar@{--}[rr]\ar[ru]\ar[rd] && 
 \diagramnode{7,10}\ar@{--}[rr]\ar[ru]\ar[rd]&&
 \diagramnode{18} \ar[rd]\\
 \diagramnode{13}\ar@{--}[rr]\ar[ru] && 
 \diagramnode{24}\ar@{--}[rr]\ar[ru] && 
 \diagramnode{35}\ar@{--}[rr]\ar[ru] && 
 \diagramnode{46}\ar@{--}[rr]\ar[ru] && 
 \diagramnode{57} \ar@{--}[rr]\ar[ru] &&
 \diagramnode{68}\ar@{--}[rr]\ar[ru] &&
 \diagramnode{79}\ar@{--}[rr]\ar[ru] &&
 \diagramnode{8,10}\ar@{--}[rr]\ar[ru] &&
 \diagramnode{19}
}
\]
The second power of $(\Gamma_{A_7},\tau)$ gives rise to three components: $(\Gamma^2_{A_3},\tau_2)$,
$(\Gamma_{(1,3)},\tau^2)\cong (\Gamma_{(2,4)},\tau^2) \cong (\Z A_4/\Sigma,\tau^2)$.
\[
\small
\xymatrix@-8mm{
 && \diagramnode{18}\ar[rdd]\ar@{--}[rr] && \diagramnode{3,10}\ar[rdd]\ar@{--}[rr] && \diagramnode{25}\ar[rdd]\ar@{--}[rr]
 && \diagramnode{47}\ar[rdd]\ar@{--}[rr] && \diagramnode{69}\ar[rdd]\ar@{--}[rr] && \diagramnode{18}\\
 & & \\
&\diagramnode{16}\ar[rdd]\ar@{--}[rr]\ar[ruu] & & \diagramnode{38}\ar[rdd]\ar[ruu]\ar@{--}[rr] & & \diagramnode{5,10}\ar[rdd]\ar[ruu]\ar@{--}[rr]
 & & \diagramnode{ 27}\ar[rdd]\ar[ruu]\ar@{--}[rr] & & \diagramnode{49}\ar[ruu]\ar@{--}[rr]\ar[rdd] &&\diagramnode{16}\ar[ruu]\\
&&\\
\diagramnode{14}\ar[ruu]\ar@{--}[rr] && \diagramnode{36}\ar[ruu]\ar@{--}[rr] && \diagramnode{58}\ar[ruu]\ar@{--}[rr]
 && \diagramnode{7,10}\ar[ruu]\ar@{--}[rr] && \diagramnode{29}\ar[ruu]\ar@{--}[rr]&&\diagramnode{14}\ar[ruu]}
\]
\[
\small
\xymatrix@-8mm{ & & \\
  & & & \diagramnode{19}\ar[rdd]\ar@{--}[rr] && \diagramnode{13}\ar[rdd]&&\\
 & & \\
  & & \diagramnode{17}\ar[ruu]\ar[rdd]\ar@{--}[rr] && \diagramnode{39}\ar[ruu]\ar@{--}[rr]\ar[rdd]&& \diagramnode{15}\ar[rdd]\\ 
  & & \\
  & \diagramnode{15}\ar[rdd]\ar[ruu]\ar@{--}[rr] & & \diagramnode{37}\ar[rdd]\ar[ruu]\ar@{--}[rr] 
  & & \diagramnode{59} \ar@{--}[rr]\ar[ruu]\ar[rdd]&& \diagramnode{17}\ar[rdd]\\ 
  & & \\
  \diagramnode{13}\ar@{--}[rr]\ar[ruu] & & \diagramnode{35}\ar[ruu]\ar@{--}[rr] & & \diagramnode{57}\ar[ruu]\ar@{--}[rr]
  & & \diagramnode{79} \ar@{--}[rr]\ar[ruu] & & \diagramnode{19}}
\small
\xymatrix@-8mm{ & & \\
  & & & \diagramnode{2,10}\ar[rdd]\ar@{--}[rr] && \diagramnode{24}\ar[rdd]&&\\
 & & \\
  & & \diagramnode{28}\ar[ruu]\ar[rdd]\ar@{--}[rr] && \diagramnode{4,10}\ar[ruu]\ar@{--}[rr]\ar[rdd]&& \diagramnode{26}\ar[rdd]\\ 
  & & \\
  & \diagramnode{26}\ar[rdd]\ar[ruu]\ar@{--}[rr] & & \diagramnode{48}\ar[rdd]\ar[ruu]\ar@{--}[rr] 
  & & \diagramnode{6,10} \ar@{--}[rr]\ar[ruu]\ar[rdd]&& \diagramnode{28}\ar[rdd]\\ 
  & & \\
  \diagramnode{24}\ar@{--}[rr]\ar[ruu] & & \diagramnode{46}\ar[ruu]\ar@{--}[rr] & & \diagramnode{68}\ar[ruu]\ar@{--}[rr]
  & &  \diagramnode{8,10} \ar@{--}[rr]\ar[ruu] & & \diagramnode{2,10}
}
\] 
By Proposition \ref{CC} one deduces that there is no $u\in\N$ such that
the two connected components isomorphic to $\Z A_4/\Sigma$
are AR quivers of a $u$-cluster category of type $A_4$.

\end{ex}

\begin{ex}

Let $m=6$ and $n=2.$ Then the connected components of the $6$th-power of $\Gamma_{A_{11}}$ are 
$\Gamma^6_{A_1}$, and the ones meeting the first three rows of  $\Gamma_{A_{11}}$.
These are  $\Gamma_{(1,3)}\cong\Gamma_{(2,4)}\cong\Gamma_{(1,4)}\cong\Gamma_{(2,5)}\cong\Z A_2/ \tau^{-4}\Sigma^2,$ which
are cylindrically shaped and
$\Gamma_{(1,5)}\cong \Gamma_{(2,6)}\cong\Z A_2/ \tau^{-2}\Sigma$ having the shape of a M\"obius band.

\[
\small
\xymatrix@-8mm{
\diagramnode{18}\ar@{--}[rrrr] &&& & \diagramnode{7,14}\ar@{--}[rrrr] &&& & \diagramnode{6,13}\ar@{--}[rrrr] &&
 & & \diagramnode{5,12}\ar@{--}[rrrr] && & & \diagramnode{4,11}\ar@{--}[rrrr] &&&&\diagramnode{3,10}
\ar@{--}[rrrr] && &&\diagramnode{29}\ar@{--}[rrrr] &&&&\diagramnode{18}
}
\]

\[
\small
\xymatrix@-8mm{
&\diagramnode{19}\ar[rdd]\ar@{--}[rr]& & \diagramnode{17}\ar[rdd]\ar@{--}[rr] & & \diagramnode{7,13}\ar[rdd]\ar@{--}[rr]
 & & \diagramnode{5,13}\ar[rdd]\ar@{--}[rr] & & \diagramnode{5,11}\ar@{--}[rr]\ar[rdd] &&\diagramnode{3,11}
\ar@{--}[rr]\ar[rdd] &&\diagramnode{39}\ar@{--}[rr]\ar[rdd] &&\diagramnode{19}\\
&&\\
\diagramnode{13}\ar[ruu]\ar@{--}[rr] && \diagramnode{79}\ar[ruu]\ar@{--}[rr] && \diagramnode{1,13}\ar[ruu]\ar@{--}[rr]
 && \diagramnode{57}\ar[ruu]\ar@{--}[rr] && \diagramnode{11,13}\ar[ruu]\ar@{--}[rr]&&\diagramnode{35}
\ar[ruu]\ar@{--}[rr]&&\diagramnode{9,11}
\ar[ruu]\ar@{--}[rr]&&\diagramnode{13}
\ar[ruu]
}
\]

\[
\small
\xymatrix@-8mm{
&\diagramnode{2,10}\ar[rdd]\ar@{--}[rr]& & \diagramnode{28}\ar[rdd]\ar@{--}[rr] & & \diagramnode{8,14}\ar[rdd]\ar@{--}[rr]
 & & \diagramnode{6,14}\ar[rdd]\ar@{--}[rr] & & \diagramnode{6,12}\ar@{--}[rr]\ar[rdd] &&\diagramnode{4,12}
\ar@{--}[rr]\ar[rdd] &&\diagramnode{4,10}\ar@{--}[rr]\ar[rdd] &&\diagramnode{2,10}\\
&&\\
\diagramnode{24}\ar[ruu]\ar@{--}[rr] && \diagramnode{8,10}\ar[ruu]\ar@{--}[rr] && \diagramnode{2,14}\ar[ruu]\ar@{--}[rr]
 && \diagramnode{68}\ar[ruu]\ar@{--}[rr] && \diagramnode{12,14}\ar[ruu]\ar@{--}[rr]&&\diagramnode{46}
\ar[ruu]\ar@{--}[rr]&&\diagramnode{10,12}
\ar[ruu]\ar@{--}[rr]&&\diagramnode{24}
\ar[ruu]
}
\]

\[
\small
\xymatrix@-8mm{
&\diagramnode{1,10}\ar[rdd]\ar@{--}[rr]& & \diagramnode{27}\ar[rdd]\ar@{--}[rr] & & \diagramnode{8,13}\ar[rdd]\ar@{--}[rr]
 & & \diagramnode{5,14}\ar[rdd]\ar@{--}[rr] & & \diagramnode{6,11}\ar@{--}[rr]\ar[rdd] &&\diagramnode{3,12}
\ar@{--}[rr]\ar[rdd] &&\diagramnode{49}\ar@{--}[rr]\ar[rdd] &&\diagramnode{1,10}\\
&&\\
\diagramnode{14}\ar[ruu]\ar@{--}[rr] && \diagramnode{7,10}\ar[ruu]\ar@{--}[rr] && \diagramnode{2,13}\ar[ruu]\ar@{--}[rr]
 && \diagramnode{58}\ar[ruu]\ar@{--}[rr] && \diagramnode{11,14}\ar[ruu]\ar@{--}[rr]&&\diagramnode{36}
\ar[ruu]\ar@{--}[rr]&&\diagramnode{9,12}
\ar[ruu]\ar@{--}[rr]&&\diagramnode{14}
\ar[ruu]
}
\]

\[
\small
\xymatrix@-8mm{
&\diagramnode{2,11}\ar[rdd]\ar@{--}[rr]& & \diagramnode{38}\ar[rdd]\ar@{--}[rr] & & \diagramnode{9,14}\ar[rdd]\ar@{--}[rr]
 & & \diagramnode{16}\ar[rdd]\ar@{--}[rr] & & \diagramnode{7,12}\ar@{--}[rr]\ar[rdd] &&\diagramnode{4,13}
\ar@{--}[rr]\ar[rdd] &&\diagramnode{5,10}\ar@{--}[rr]\ar[rdd] &&\diagramnode{2,11}\\
&&\\
\diagramnode{25}\ar[ruu]\ar@{--}[rr] && \diagramnode{8,11}\ar[ruu]\ar@{--}[rr] && \diagramnode{3,14}\ar[ruu]\ar@{--}[rr]
 && \diagramnode{6,9}\ar[ruu]\ar@{--}[rr] && \diagramnode{1,12}\ar[ruu]\ar@{--}[rr]&&\diagramnode{47}
\ar[ruu]\ar@{--}[rr]&&\diagramnode{10,13}
\ar[ruu]\ar@{--}[rr]&&\diagramnode{25}
\ar[ruu]
}
\]

\[
\small
\xymatrix@-8mm{ 
  & \diagramnode{1,11}\ar[rdd]\ar@{--}[rr] & & \diagramnode{37}\ar[rdd]\ar@{--}[rr] 
  & & \diagramnode{9,13} \ar@{--}[rr]\ar[rdd]&& \diagramnode{15}\ar[rdd]\\ 
  & & \\
  \diagramnode{15}\ar@{--}[rr]\ar[ruu] & & \diagramnode{7,11}\ar[ruu]\ar@{--}[rr] & & \diagramnode{3,13}\ar[ruu]\ar@{--}[rr]
  & &  \diagramnode{59} \ar@{--}[rr]\ar[ruu] & & \diagramnode{1,11}
}
\]

\[
\small
\xymatrix@-8mm{ 
  & \diagramnode{2,12}\ar[rdd]\ar@{--}[rr] & & \diagramnode{48}\ar[rdd]\ar@{--}[rr] 
  & & \diagramnode{10,14} \ar@{--}[rr]\ar[rdd]&& \diagramnode{26}\ar[rdd]\\ 
  & & \\
  \diagramnode{26}\ar@{--}[rr]\ar[ruu] & & \diagramnode{8,12}\ar[ruu]\ar@{--}[rr] & & \diagramnode{4,14}\ar[ruu]\ar@{--}[rr]
  & &  \diagramnode{6,10} \ar@{--}[rr]\ar[ruu] & & \diagramnode{2,12}
}
\]

Furthermore, we deduce from Proposition \ref{CC} that $\Gamma_{(1,3)}$, $\Gamma_{(2,4)}$ and 
$\Gamma_{(1,4)}$, $\Gamma_{(2,5)}$ are isomorphic to the AR quiver of a $4$-cluster category of type $A_2$.
\end{ex}

\end{document}